\documentclass[12pt, a4paper]{amsproc}
\usepackage{amsmath,amssymb,anysize,times,float,enumerate}
\usepackage{hyperref} 
\hypersetup{citecolor=red, linkcolor=blue, colorlinks=true}

\newtheorem{thm}{Theorem}[section]

\newtheorem{conj}[thm]{Conjecture}
\newtheorem{lem}[thm]{Lemma}
\newtheorem{obs}[thm]{Observation}
\newtheorem{prop}[thm]{Proposition}

\def\Sz{{\rm Sz}}

\def\GF{{\rm GF}}
\def\GAP{\textbf{{\rm GAP}}}

\def\Out{{\rm Out}}
\def\GL{{\rm GL}}

\newcommand{\norml}{\vartriangleleft}
\newcommand{\liso}{\lesssim}

\renewcommand\le{\leqslant}
\renewcommand\ge{\geqslant}

\title[Point regular generalized quadrangles]{On generalized quadrangles with a point regular group of automorphisms}

\author{Eric Swartz}

\address{Department of Mathematics, College of William and Mary, P.O. Box 8795, Williamsburg, VA 23187-8795}
\email{easwartz@wm.edu}



\begin{document}

\begin{abstract}
 A generalized quadrangle is a point-line incidence geometry such that any two points lie on at most one line and, given a line $\ell$ and a point $P$ not incident with $\ell$, there is a unique point of $\ell$ collinear with $P$.  We study the structure of groups acting regularly on the point set of a generalized quadrangle.  In particular, we provide a characterization of the generalized quadrangles with a group of automorphisms acting regularly on both the point set and the line set and show that such a thick generalized quadrangle does not admit a polarity.  Moreover, we prove that a group $G$ acting regularly on the point set of a generalized quadrangle of order $(u^2, u^3)$ or $(s,s)$, where $s$ is odd and $s+1$ is coprime to $3$, cannot have any nonabelian minimal normal subgroups.
\end{abstract}

\maketitle

\section{Introduction}
A \textit{finite generalized quadrangle} is an incidence geometry consisting of a finite set $\mathcal{P}$ of points and a finite set $\mathcal{L}$ of lines such that, if $P \in \mathcal{P}$ and $\ell \in \mathcal{L}$ such that $P$ is not incident with $\ell$, then there is a unique point on $\ell$ collinear with $P$.  If every point is incident with at least three lines and every line is incident with at least three points, then the generalized quadrangle is said to be \textit{thick}.  From this definition, one can show that in a thick generalized quadrangle any point is incident with the same number of lines and any line is incident with the same number of points.  A generalized quadrangle where every point is incident with exactly $t+1$ lines and any line is incident with exactly $s+1$ points, where $s$ and $t$ are fixed positive integers, is said to have \textit{order} $(s,t)$.  For further basic properties of generalized quadrangles, see \cite{fgq}. 

A group $G$ is said to act \textit{regularly} on a set $\Omega$ if $G$ is transitive on $\Omega$ and, for all $\omega \in \Omega$, the stabilizer $G_\omega$ of $\omega$ in $G$ is trivial.  When a group $G$ acts regularly on a set $\Omega$, there is a natural identification of elements of $G$ with elements of $\Omega$: if we choose a distinguished element $\omega \in \Omega$, then we may identify $\omega$ with the identity $1 \in G$, and, if $\alpha \in \Omega$, we may identify $\alpha$ with the unique element $g \in G$ such that $\omega^g = \alpha$.  

There has been substantial interest in generalized quadrangles with a group of automorphisms $G$ acting regularly on the point set $\mathcal{P}$, and we briefly summarize here the current state of knowledge.  Such generalized quadrangles were first studied in \cite{ghinelli}, where techniques from the study of difference sets of groups are used to show that, if $\mathcal{Q}$ is a generalized quadrangle of order $(s,t)$, where $s = t$ and $s$ is even, and $G$ acts regularly on the point set of $\mathcal{Q}$, then $G$ has trivial center, $G$ is not a Frobenius group, and $s^2 + 1$ is not squarefree.  These techniques were pushed even further in \cite{yoshiara}, where it is shown that a generalized quadrangle with $s = t^2$ cannot have a group of automorphisms that acts regularly on points.  In \cite{DT1}, it is shown that if $\mathcal{Q}$ is a thick finite generalized quadrangle of order $(s,t)$ admitting an abelian automorphism group that acts regularly on points, then $\mathcal{Q}$ is isomorphic to a generalized quadrangle $T_2^*(O)$ arising from a generalized hyperoval, and the group acting on $\mathcal{Q}$ is elementary abelian of order $2^{3n}$ for some natural number $n$.  (For further background information and descriptions of the generalized quadrangles referred to in this paragraph, the curious reader is again directed to \cite{fgq}.)  In \cite{DT2}, it is shown that if a Heisenberg group of order $q^3$, where $q$ is odd, acts regularly on the point set of $\mathcal{Q}$, then $\mathcal{Q}$ is a \textit{Payne-derived generalized quadrangle} of a thick \textit{elation generalized quadrangle} having a \textit{regular point}.  (A regular point of a generalized quadrangle is defined combinatorially and has no relation to a regular group action.)  It was shown in \cite{BambergGiudici} that the only classical generalized quadrangles with a point-regular group of automorphisms are $Q(5,2)$ and $Q(5,8)$, and the authors showed that the Payne-derived generalized quadrangles of order $(q-1, q+1)$ obtained from the \textit{symplectic generalized quadrangle} $W(q)$, where $q$ is a prime power but not prime, can actually have several distinct groups of automorphisms acting regularly on the point set, showing that the class of such groups is much more varied than previously thought.  Finally, the study of groups acting regularly on the point set of generalized quadrangles figures prominently in the study of \textit{skew-translation generalized quadrangles}; see \cite{ghinelliAS} and the subsequent work in \cite{ASConfigs}.

The purpose of this paper is to study further the structure of such generalized quadrangles and the groups that act regularly on their point sets.  A first step in the study of a generalized quadrangle $\mathcal{Q}$ with a point regular group $G$ of automorphisms is to consider the action of $G$ on the lines of $\mathcal{Q}$.  One possibility of interest is that $G$ also acts regularly on lines, since, otherwise, some elements in the group acting regularly on the point set must fix a line, and one can proceed with an analysis of line stabilizers.  In the case when $G$ also acts regularly on lines, we have $s = t$, in which case we say that $\mathcal{Q}$ has order $s$ (see Lemma \ref{lem:GQbasics} (i)).  The only known example of a generalized quadrangle with a group of automorphisms acting regularly on both points and lines is the unique thin generalized quadrangle of order $(1,1)$, and it is not difficult to show that $s$ would have to be even in any other examples (see Section \ref{sect:lines}).

We are able to characterize further the groups that act regularly on both the point set and the line set of a generalized quadrangle.  To that end, we have the following result, which follows immediately from Propositions \ref{prop:regregGQ} and \ref{prop:regregGQ2}.

\begin{thm}
 \label{thm:regregchar}
 The group $G$ acts regularly on both the point set and the line set of a generalized quadrangle $\mathcal{Q}$ of order $s$ if and only if $G$ contains a subset $\Sigma = \{g_0 = 1, g_1, g_2,..., g_s\}$ satisfying the following two conditions:
\begin{itemize}
 \item[(AX1)] If $g \in G \backslash \Sigma$, then there exist unique integers $i,j,k$ (with $i \neq j$ and $k \neq j$) such that $g = g_ig_j^{-1}g_k$.
 \item[(AX2)] If $g_ig_j^{-1}g_k \in \Sigma,$ then either $i = j$ or $j = k$.
\end{itemize} 
\end{thm}

Since there is a natural identification of $G$ with both the point set and the line set in this case, it is natural to think that $\mathcal{Q}$ should have a polarity $\theta$; that is, an order two isomorphism from $\mathcal{Q}$ to the \textit{dual} generalized quadrangle $\mathcal{Q}'$, which has point set $\mathcal{L}$ and line set $\mathcal{P}$.  In fact, the following result shows that such a thick generalized quadrangle cannot have a polarity.

\begin{thm}
 \label{thm:polarity}
 A generalized quadrangle $\mathcal{Q}$ of order $s > 1$ that has a group of automorphisms acting regularly on both the point set and line set cannot admit a polarity.
\end{thm}

On the other hand, the cyclic group of order $4$ satisfies the conditions listed in Theorem \ref{thm:regregchar} (see Observation \ref{obs:thin}).  This begs the question of whether there exist any groups of order greater than $4$ satisfying the conditions in Theorem \ref{thm:regregchar}.  Toward that end, we prove the following theorem, which places severe restrictions on the order of such a generalized quadrangle and the group acting regularly on its point set and line set.  (For information on the group theoretic notation used in the statement of Theorem \ref{thm:regregorder}, see Subsection \ref{subsub:group}.)

\begin{thm}
 \label{thm:regregorder}
Let the group $G$ act regularly on both the point set and the line set of a thick generalized quadrangle of order $s$.  Then, $G$ is solvable, $s+1$ is coprime to both $2$ and $3$, and $s^2 + 1$ is not squarefree.  Moreover, there exists a unique prime $p$ dividing $s^2 + 1$ such that $F(G) = O_p(G)$ and the following conditions are satisfied:
\begin{itemize}
 \item[(1)] $F(G)$ is not cyclic;
 \item[(2)] $|F(G)| \ge 2s + 3$;
 \item[(3)] If $p^n$ is the highest power of $p$ dividing $|G|$, then $|G|$ divides $p^n \cdot \prod_{k=0}^{n-1}(p^{n-k} - 1)$.
\end{itemize}
\end{thm}

We are able to use Theorem \ref{thm:regregorder} to show that, if $G$ acts regularly on the point set and the line set of a thick generalized quadrangle of order $s$, then $s > 10^8$ and $|G| > 10^{24} +10^{16} + 10^8 + 1$ (see Proposition \ref{prop:numerical}).  Ghinelli \cite{ghinelli} conjectures that no generalized quadrangle of order $s$, when $s$ is even, has a group of automorphisms acting regularly on points.  Based on Theorem \ref{thm:polarity} and the numerical evidence, we make a more modest conjecture.

\begin{conj}
\label{conj:regreg}
 A generalized quadrangle $\mathcal{Q}$ of order $s > 1$ cannot have a group of automorphisms acting regularly on both the point set and the line set.
\end{conj}

In this paper, we also study the structure of groups that could act regularly on the point set of a generalized quadrangle, leading to the following result.

\begin{thm}
\label{thm:minnormalsubs}
 Let $\mathcal{Q}$ be a generalized quadrangle of order $(u^2, u^3)$ or $(s,s)$, where $u > 1$ and $s$ is odd with $s+1$ coprime to $3$.  If a group $G$ acts regularly on the point set of $\mathcal{Q}$, then $G$ does not have any nonabelian minimal normal subgroups.
\end{thm}

It should be noted that Theorem \ref{thm:minnormalsubs} refines part of a result of Yoshiara (see Lemma \ref{lem:yoshiara} (iv)) and is an immediate consequence of Theorems \ref{thm:s} and \ref{thm:u^2}, which consider, respectively, the generalized quadrangles of order $s$ and the generalized quadrangles of order $(u^2, u^3)$.  It is natural to ask about these orders in particular, since there exist infinite families of these orders; see \cite[Chapter 3]{fgq}

This paper is organized as follows.  Section \ref{sect:background} contains background information and sets the stage for the results proved later in this paper.  In Section \ref{sect:lines}, we prove Theorems \ref{thm:regregchar}, \ref{thm:polarity} and \ref{thm:regregorder}.  In Section \ref{secss} we study the properties of a group $G$ that acts regularly on the point set of a generalized quadrangle of order $s$, where $s$ is odd and $s+1$ is coprime to $3$.  Finally, in Section \ref{sect:u^2} we study the properties of a group $G$ that acts regularly on the point set of a generalized quadrangle of order $(u^2, u^3)$, where $u > 1$.

\section{Background}
\label{sect:background}

\subsection{Group theory}
\label{subsub:group}
We begin with group theoretic terminology and notation.  In a finite group $G$, the \textit{Fitting subgroup} is the largest normal nilpotent subgroup of $G$ and is denoted $F(G)$.  Given a prime $p$, the \textit{p-core} $O_p(G)$ of $G$ is the largest normal $p$-subgroup of $G$.  The Fitting subgroup is also the product of the $p$-cores of $G$ for all primes $p$ dividing $|G|$.  A \textit{quasisimple} group is a perfect central extension of a simple group, and a \textit{component} of a group is a subnormal quasisimple group.  The \textit{layer} $E(G)$ of a group $G$ is the subgroup generated by all components.  The \textit{generalized Fitting subgroup} $F^*(G)$ is the subgroup generated by the Fitting subgroup and the layer.  In a solvable group, the generalized Fitting subgroup is the same as the Fitting subgroup.  Each of the Fitting subgroup and the generalized Fitting subgroup contains its own centralizer.  For more information regarding these concepts, see \cite{Aschbacher}. 

We introduce now basic information about the \textit{Suzuki groups}, an infinite family of finite simple groups figuring prominently in later sections of this paper, since they are the only finite simple groups whose orders are coprime to $3$ \cite[Chapter II, Corollary 7.3]{Glauberman}.  The following omnibus lemma contains results about the Suzuki groups $\Sz(q)$, where $q = 2^{2m+1}$ for some natural number $m$.  

\begin{lem}
 \label{lem:suz}
 Let $m$ be a natural number, $q = 2^{2m+1}$, and $G = \Sz(q)$.
 \begin{itemize}
  \item[(i)] \cite[Theorem 7]{suzuki} $|G| = q^2(q-1)(q^2 + 1)$.
  \item[(ii)] \cite[Theorem 3.6]{fingps} $3$ does not divide $|G|$.
  \item[(iii)] \cite[Theorem 3.9]{fingps} If $p$ is an odd prime, then the Sylow $p$-subgroups of $G$ are cyclic.
  \item[(iv)] \cite[Theorem 4.1]{fsg} The maximal subgroups of $G$ are isomorphic to one of the following:
	  \begin{itemize}
	    \item[(a)] $N_G(Q)$, where $Q$ is a Sylow $2$-subgroup of $G$ and $|N_G(Q)| = q^2(q-1)$;
	    \item[(b)] $D_{q-1}$, a dihedral group of order $2(q-1)$;
	    \item[(c)] $C_{q + \sqrt{2q} + 1}:C_4$;
	    \item[(d)] $C_{q - \sqrt{2q} + 1}:C_4$;
	    \item[(e)] $\Sz(q_0)$, where $q = q_0^r$ for some natural number $r$ and $q_0 > 2$.
	  \end{itemize}  
 \item[(v)] The order of the smallest centralizer of an element in $G$ is $q - \sqrt{2q} + 1$, and the size of the largest conjugacy class is $q^2(q-1)(q + \sqrt{2q} + 1)$.
 \item[(vi)] \cite[2.2.5]{fangphd} The order of the centralizer of an element of order $4$ in $G$ is $2q$.
 \item[(vii)] \cite[Section 4.2.4]{fsg} $G$ has trivial Schur multiplier unless $m = 1$, in which case the Schur multiplier is elementary abelian of order $4$.
 \item[(viii)] \cite[Theorem 11]{suzuki} The outer automorphism group of $G$ is isomorphic to the Galois group of $\GF(q)$ and is a cyclic group of order $2m+1$.
 \end{itemize}
\end{lem}

Note that Lemma \ref{lem:suz} (v) follows from examination of the subgroups listed in Lemma \ref{lem:suz} (iv) and noting that a Sylow $2$-subgroup $Q$ of $G$ has a center of size $q$.  For more about these groups, the reader is referred to the original work of Suzuki \cite{suzuki} or the more recent collection of results in \cite[Section 3]{localsuz}.

\subsection{Generalized quadrangles}

We now discuss some results about generalized quadrangles.  Let $\mathcal{Q}$ be a finite generalized quadrangle with point set $\mathcal{P}$, line set $\mathcal{L}$, and order $(s,t)$.  Given $P, Q \in \mathcal{P}$, the notation $P \sim Q$ indicates that $P$ and $Q$ are distinct collinear points.  The following lemma concerns the parameters $s$ and $t$.

\begin{lem}\cite[1.2.1, 1.2.2, 1.2.3]{fgq}
\label{lem:GQbasics} The following hold:
\begin{itemize}
 \item[(i)] $|\mathcal{P}| = (s+1)(st + 1)$ and $|\mathcal{L}| = (t+1)(st+1)$;
 \item[(ii)] $s + t$ divides $st(s+1)(t+1)$;
 \item[(iii)] $t \le s^2$ and $s \le t^2$.
\end{itemize}
\end{lem}

The first paper to consider a group $G$ acting regularly on the point set $\mathcal{P}$ of a generalized quadrangle $\mathcal{Q}$ was \cite{ghinelli}.  Note that, if a group $G$ acts regularly on the point set of a generalized quadrangle of order $s$, where $s$ is even, then $|G| = (s+1)(s^2+1)$ is odd, and so $G$ is solvable.  The following lemma collects various results proved by Ghinelli that will be useful later.  

\begin{lem}\cite[Theorem 3.5, Theorem 4.2, Theorem 4.3, Theorem 4.6]{ghinelli}
\label{lem:ghinelli}
 Let $G$ be an automorphism group acting regularly on the point set of a generalized quadrangle of order $s$.  Then the following hold:
 \begin{itemize}
  \item[(i)] If $N$ is an elementary abelian normal $p$-subgroup of $G$, $p \neq 2$, then $|N|$ divides $s^2 + 1$.  In particular, if $p > 2$ and $O_p(G) > 1$, then $p$ divides $s^2 + 1$.
  \item[(ii)] If $s$ is even, then $G$ is solvable, $F(G)$ is a $p$-group with $|F(G)| \ge s+1$, and $F(G)$ is not cyclic.
  \item[(iii)] If $s$ is even, then $Z(G)$ is trivial and $s^2 + 1$ is not squarefree.
 \end{itemize}

\end{lem}

Yoshiara was able to generalize many of the results of Ghinelli in \cite{yoshiara}.  The following lemma collects various results proved by Yoshiara that will be useful later in this paper.

\begin{lem}\cite[Lemma 4, Lemma 6, Lemma 7, Theorem 8, Lemma 9, Lemma 10]{yoshiara}
 \label{lem:yoshiara}
 Let $G$ be a group that acts regularly on the point set of a generalized quadrangle of order $(s,t)$.  Assume that $\gcd(s,t) > 1$, and let $P$ be a distinguished point of the generalized quadrangle.  Define $\Delta:= \{g \in G : P^g \sim P \} \cup \{1\}$, and let $\Delta^c$ denote the complement of $\Delta$ in $G$.
 \begin{itemize}
  \item[(i)] If $H \le G$ and $H$ is contained entirely in $\Delta$, then there is a unique line $\ell$ through $P$ such that $H \le G_\ell$.
  \item[(ii)] If $x$ is a nontrivial element of $G$ and $x^G$ denotes the conjugacy class of $x$ in $G$, then $x^G \cap \Delta \neq \varnothing$ and $|x^G \cap \Delta^c|$ is a multiple of $\gcd(s,t)$ (possibly equal to $0$). 
  \item[(iii)] If $p = 2, 3$, then the following are equivalent:
      \begin{itemize}
	\item[(a)] $|G|$ is divisible by $p$.
	\item[(b)] There is an element of order $p$ fixing a line through $P$.
	\item[(c)] $s+1$ is divisible by $p$.
      \end{itemize}
   \item[(iv)] If $s+1$ is coprime to $3$ and $G$ has a nonabelian minimal normal subgroup $N$, then $N$ is the internal direct product $S_1S_2 \dots S_m$ and each $S_i \cong \Sz(q)$ for a fixed $q = 2^{2e+1}$.    
   \item[(v)] If $N$ is a nontrivial normal subgroup of $G$ such that $N$ is entirely contained in $\Delta$, then $|N|$ divides $\gcd(s+1, t^2 - t)$.
   \item[(vi)] If the conjugacy class $x^G$ of a nontrivial element $x \in G$ is contained entirely in $\Delta$, then $|x|$ divides $s+1$. 
 \end{itemize}
\end{lem}

The following lemma is extremely useful and is used repeatedly in later sections.

\begin{lem}
 \label{lem:xGc}
 Let $G$ be a group that acts regularly on the point set of a generalized quadrangle of order $(s,t)$.  Assume that $\gcd(s,t) > 1$, and let $P$ be a distinguished point of the generalized quadrangle.  Define $\Delta:= \{g \in G : P^g \sim P \} \cup \{1\}$, and let $\Delta^c$ denote the complement of $\Delta$ in $G$.  For any $x \in G$, if $x^G \cap \Delta^c \neq \varnothing$, then $|x^G| \ge \gcd(s,t) + 1$.
\end{lem}

\begin{proof}
 Assume that $G$ and $x$ are as in the statement of the lemma.  By Lemma \ref{lem:yoshiara} (ii), $|x^G \cap \Delta| \ge 1$ and $|x^G \cap \Delta^c| \equiv 0 \pmod {\gcd(s,t)}$.  Since $|x^G \cap \Delta^c| > 0$ by assumption, the result follows.
\end{proof}

\section{Generalized quadrangles with a group of automorphisms acting regularly on both points and lines}
\label{sect:lines}

Throughout this section, we will assume that $\mathcal{Q}$ is a generalized quadrangle of order $(s,t)$ with a group of automorphisms $G$ that acts regularly on both the point set $\mathcal{P}$ and the line set $\mathcal{L}$.  This immediately implies that $s = t$ and $|G| = (s+1)(s^2+1)$.  

We now prove two propositions that characterize the generalized quadrangles with a group that acts regularly on both points and lines.

\begin{prop}
\label{prop:regregGQ}
Let $\Sigma = \{g_0 = 1, g_1, g_2,..., g_s\}$ be a subset of a group $G$ satisfying the following:
\begin{itemize}
 \item[(AX1)] If $g \in G \backslash \Sigma$, then there exist unique integers $i,j,k$ (with $i \neq j$ and $k \neq j$) such that $g = g_ig_j^{-1}g_k$.
 \item[(AX2)] If $g_ig_j^{-1}g_k \in \Sigma,$ then either $i = j$ or $j = k$.
\end{itemize}
Then there exists a generalized quadrangle $\mathcal{Q}$ such that $\mathcal{Q}$ has order $s$ and $G$ acts regularly on both the point set $\mathcal{P}$ and the line set $\mathcal{L}$ of $\mathcal{Q}$. 
\end{prop}

\begin{proof}
First, there are precisely $s+1$ elements in $\Sigma$, and $g_ig_j^{-1}g_k \in \Sigma$ only when $i = j$ or $j = k$ by (AX2).  Moreover, the manner in which $g \in G \backslash \Sigma$ can be written as a product $g_ig_j^{-1}g_k$ is unique by (AX1).  This implies that there are $s^2(s+1)$ elements in $G \backslash \Sigma$, since there are $(s+1)$ choices for $j$ and $s$ choices for each of $i$ and $k$, and so $|G| = (s+1) + s^2(s+1) = (s+1)(s^2+1)$.

We define the point set $\mathcal{P}$ to be the elements of the group $G$ (where the point associated with the group element $g$ will be denoted $P_g$), and we define the line set $\mathcal{L}$ to be the sets $\Sigma h$, where $h \in G$ (we denote the line associated with the set $\Sigma h$ by $\ell_h$).  Finally, the point $P_x$ will be incident with the line $\ell_y$ if and only if $xy^{-1} \in \Sigma$.

For any $g \in G$, note that the point $P_x$ is incident with the line $\ell_y$ if and only if $(xg)(yg)^{-1} = xy^{-1} \in \Sigma$ if and only if $P_{xg}$ is incident with $\ell_{yg}$.  Hence each $g \in G$ acts as a collineation of the putative generalized quadrangle $\mathcal{Q}$, and this action is regular on both points and lines. 

Since $|\Sigma| = s+1$, each line contains $s+1$ points.  Moreover, if a point $P_x$ is on a line $\ell_y$, then $xy^{-1} \in \Sigma$.  For a fixed $x \in G$, there are exactly $s+1$ elements $y$ of $G$ such that $xy^{-1} \in \Sigma$, so each point is on exactly $s+1$ lines.

Let $P_x$ and $P_y$ be distinct points of $\mathcal{Q}$, and assume that both points are on the lines $\ell_g$ and $\ell_h$.  This means that there exist integers $i,j,m,n$ such that $x = g_ih = g_mg$ and $y = g_jh = g_ng$.  Thus $xy^{-1} = g_ig_j^{-1} = g_mg_n^{-1}$, which implies that $g_i = g_mg_n^{-1}g_j \in \Sigma$.  By (AX2), this means that either $g_n = g_m$ or $g_n = g_j$.  If $g_n = g_m$, then $y = g_ng = g_mg = x$, a contradiction to distinctness.  Thus $g_n = g_j$, which implies that $g = h$.  Therefore, distinct points are incident with at most one line.

Let $\ell_x$ and $\ell_y$ be distinct lines of $\mathcal{Q}$ and assume that they intersect in at least one point.  This means that $g_ix = g_jy$ for some $0 \le i,j \le s,$ and furthermore $xy^{-1} = g_i^{-1}g_j$.  Assume first that $xy^{-1} \notin \Sigma.$  This means that $xy^{-1} = g_0g_i^{-1}g_j$, and by (AX1), $i$ and $j$ are unique.  In this case, the lines $\ell_x$ and $\ell_y$ meet in exactly one point.  Assume now that $xy^{-1} \in \Sigma.$  The lines $\ell_x$ and $\ell_y$ are distinct, so $xy^{-1} \neq 1 = g_0$.  Thus $xy^{-1} = g_j$ for some $j >0$, and $x = g_jy$.  Note that the size of the intersection of $\Sigma g_jy$ and $\Sigma y$ is the same as the size of the intersection of $\Sigma g_j$ and $\Sigma$.  Suppose that $g_i$ and $g_k$ are such that $g_ig_j = g_k$.  Then $g_j = g_i^{-1}g_k = g_0g_i^{-1}g_k$.  By (AX2), either $g_i = 1$ or $g_i = g_k$.  If $g_i = 1$, then $g_k = g_j$, and the two lines intersect in a unique point.  If $g_i = g_k$, then $g_j = 1$, a contradiction to the distinctness of the lines.  In any case, two distinct lines are mutually incident with at most one point.

Let $P_x$ be a point that is not incident with the line $\ell_y$.  This is equivalent to $xy^{-1} \notin \Sigma.$  By (AX1), there exist unique $i,j,k$ such that $xy^{-1} = g_ig_j^{-1}g_k$, i.e., there exists a unique $i$ such that $\Sigma y \cap \Sigma g_i^{-1}x$ is nonempty.  Note that the lines through the point $P_1$ are precisely the lines $\ell_{g_i^{-1}}$; hence the lines through $P_x$ are precisely the lines $\ell_{g_i^{-1}x}$.  Hence there is a unique line $\ell_{g_i^{-1}x}$ through $P_x$ such that $\ell_{g_i^{-1}x}$ and $\ell_y$ are incident with a common point.  Therefore, $\mathcal{Q}$ is a generalized quadrangle such that $G$ acts regularly on both the points and lines of $\mathcal{Q}$, as desired.   
\end{proof}

\begin{prop}
\label{prop:regregGQ2}
Let $G$ be a group that acts regularly on both the point set $\mathcal{P}$ and the line set $\mathcal{L}$ of a finite generalized quadrangle $\mathcal{Q}$ of order $s$.  Then there exists a subset $\Sigma$ of elements of $G$ satisfying (AX1) and (AX2) of Proposition \ref{prop:regregGQ}. 
\end{prop}

\begin{proof}
Assume that such a generalized quadrangle $\mathcal{Q}$ exists.  Fix a point $P_1 \in \mathcal{P}$ and identify it with the identity element $1 \in G$.  Since $G$ acts regularly on $\mathcal{P}$, for each $x \in G$, we may define $P_x := P_1^x$.  Let $\ell$ be a line incident with $P_1$, define $\Sigma:= \{g \in G : P_g \text{ is incident with }\ell \}$, and label the points of $\Sigma$ such that $\Sigma = \{g_0 = 1, g_1,..., g_s \}$.  Note that $G$ acts regularly on the lines of $\mathcal{Q}$, so we may identify the lines of $\mathcal{Q}$ with the sets of elements of the form $\Sigma h$, where $h \in G$.  Without a loss of generality we may denote the line $\Sigma h$ by $\ell_h$, and so the line $\ell$ is identified as $\ell_1$.

Assume first that $g \notin \Sigma.$  This means that $P_g$ is not incident with $\ell_1$.  Since $\mathcal{Q}$ is a generalized quadrangle, there exists a unique point $P_{g_k}$ on $\ell_1$ such that $P_g$ is collinear with $P_{g_k}$.  Since $G$ acts regularly on both points and lines, $P_1$ is incident with precisely the lines $\ell_{g_j^{-1}}$, $0 \le j \le s$, which implies that the lines incident with $P_{g_k}$ are of the form $\ell_{g_j^{-1}g_k}.$  The point $P_g$ is on a unique line of that form, and so $g$ may be written uniquely as $g = g_ig_j^{-1}g_k$, and $\Sigma$ satisfies (AX1).

Now suppose that the product $g_ig_j^{-1}g_k \in \Sigma$, and assume that $i \neq j$.  Since $P_{g_ig_j^{-1}g_k}$ is incident with $\ell_1$, we have that $P_{g_ig_j^{-1}}$ is incident with line $\ell_{g_k^{-1}}$.  Since $i \neq j$, $g_ig_j^{-1} \neq 1$, so $P_{g_ig_j^{-1}}$ is distinct from $P_1$ but collinear with $P_1$.  On the other hand, $P_{g_ig_j^{-1}}$ is also on line $\ell_{g_j^{-1}}$, which is also incident with $P_1$.  Since two points are incident with at most one line, we conclude that $\ell_{g_j^{-1}} = \ell_{g_k^{-1}}$.  Therefore, $j = k$, and $\Sigma$ satisfies (AX2), as desired.      
\end{proof}

We can now prove Theorem \ref{thm:regregchar}.

\begin{proof}[Proof of Theorem \ref{thm:regregchar}]
This follows immediately from Propositions \ref{prop:regregGQ} and \ref{prop:regregGQ2}.
\end{proof}

\begin{obs}
\label{obs:thin}
The cyclic group $C_4$ satisfies the conditions of Proposition \ref{prop:regregGQ}. 
\end{obs}

Indeed, if $C_4 = \langle x \rangle$, then the set $\Sigma = \{1, x \}$ satisfies all three conditions, and hence the unique thin generalized quadrangle of order $(1,1)$ has a group of automorphisms that acts regularly on both its point set and its line set.

Henceforth in this section, we fix a point $P \in \mathcal{P}$, and we define $\Delta:= \{g \in G : P^g \sim P \} \cup \{1\}.$  

\begin{lem}
\label{lem:orderGreg}
If $s > 1$, then $|G|$ is coprime to both $2$ and $3$. 
\end{lem}

\begin{proof}
This follows directly from Lemma \ref{lem:yoshiara} (iii).
\end{proof}

\begin{lem}
\label{lem:FGregreg}
If $s > 1$, then $G$ is solvable and $F(G) = O_p(G)$, where $p$ is an odd prime dividing $s^2 + 1$. 
\end{lem}

\begin{proof}
This follows immediately from Lemma \ref{lem:orderGreg} and Lemma \ref{lem:ghinelli} (iii).
\end{proof}

\begin{lem}
\label{lem:Opsize}
If $s > 1$ and $F(G) = O_p(G)$, then $|O_p(G)| \ge 2s+3$. 
\end{lem}

\begin{proof}
Let $1 \neq x \in O_p(G).$  Suppose that $x$ is conjugate in $G$ to $x^{-1}$.  Then there exists $g \in G$ such that $x^g = x^{-1}$, and so $x^{g^2} = x$.  Since $|G|$ has odd order by Lemma \ref{lem:orderGreg}, $g$ has odd order, and there is some $n \in \mathbb{N}$ such that $g = (g^2)^n$, which means that $g$ centralizes $x$.  However, this means that $x = x^{-1}$, and $x$ must have order dividing $2$, a contradiction.  Hence $x$ and $x^{-1}$ cannot be conjugate in $G$, and $O_p(G)$ has at least three conjugacy classes: $1^G$, $x^G$, and $(x^{-1})^G$.

By Lemma \ref{lem:yoshiara} (vi) we know that neither $x^G$ nor $(x^{-1})^G$ is contained entirely in $\Delta$.  Therefore, by Lemma \ref{lem:xGc}, \[|O_p(G)| \ge |1^G| + |x^G| + |(x^{-1})^G| \ge 1 + (s+1) + (s+1) = 2s+3,\] as desired. 
\end{proof}

We are now ready to prove Theorem \ref{thm:polarity}.

\begin{proof}[Proof of Theorem \ref{thm:polarity}]
Assume that $\mathcal{Q}$ has a polarity.  By \cite[1.8.2]{fgq}, $2s$ is a square, which means that $s^2 + 1$ factors into $(s + \sqrt{2s} + 1)(s - \sqrt{2s} + 1)$.  By Lemmas \ref{lem:FGregreg} and \ref{lem:Opsize}, there exists a normal subgroup $O_p(G)$ of order $p^d \ge 2s+3$, where $p$ divides $s^2+1$.  We already know that $p$ is odd by Lemma \ref{lem:orderGreg}, and, since $p$ divides $s^2 + 1$, $p$ is coprime to $s$, which implies that $p^d$ divides either $s + \sqrt{2s} + 1$ or $s - \sqrt{2s} + 1$, but not both.  However, for any $s > 1$, $s - \sqrt{2s} + 1 < s + \sqrt{2s} + 1 < 2s+3 \le p^d$, a contradiction.  Therefore, $\mathcal{Q}$ does not admit a polarity.    
\end{proof}

Note that Theorem \ref{thm:polarity} implies that the regular action of $G$ on the set of points is necessarily different than the action of $G$ on the set of lines, in the sense that the set $\Delta$ defined above must be distinct from $\Delta_\mathcal{L}$, where for some fixed line $\ell$ incident with the fixed point $P$, 
\[\Delta_\mathcal{L} := \{g \in G: \ell^g \text{ is distinct from and concurrent with } \ell\} \cup \{1\}.\]

This begs the question of whether a group can possibly act regularly on both the point set or the line set of a thick generalized quadrangle.  Toward that end, we have the following proposition, which is the last piece before we can prove Theorem \ref{thm:regregorder}.

\begin{prop}
 \label{prop:regregorder}
 Let $G$ act regularly on both the point set and the line set of a generalized quadrangle of order $s$.  If $p$ is the unique odd prime dividing $s^2 + 1$ such that $F^*(G) = F(G) = O_p(G)$ and $p^n$ is the highest power of $p$ dividing $|G|$, then $|G|$ divides $p^n \cdot \prod_{k=0}^{n-1}(p^{n-k} - 1)$.
\end{prop}

\begin{proof}
 We note that $G$ is solvable, and $F(G) = O_p(G)$ for some prime $p$ dividing $s^2 + 1$ by Lemma \ref{lem:FGregreg}.  Let $p^n$ be the order of a Sylow $p$-subgroup of $G$, and let $H$ be a Hall $p'$-subgroup of $G$, which exists by Hall's Theorem \cite[Theorem 6.4.1]{Gorenstein}.  Since $F(G) = O_p(G)$, $F(G)$ is self-centralizing, and, since $H$ is coprime to $p$, by \cite[Theorem 5.3.5]{Gorenstein} we have that $H$ acts faithfully on the Frattini quotient $F(G)/\Phi(F(G))$, which is elementary abelian of rank at most $n$.  Thus $H \liso \GL(n,p)$ with order coprime to $p$, and, since $|G| = p^n|H|$, the result follows.  
\end{proof}

\begin{proof}[Proof of Theorem \ref{thm:regregorder}.]
By Lemma \ref{lem:orderGreg}, $|G|$ is not divisible by $2$ or $3$, and by Lemma \ref{lem:yoshiara} (iii), this means $s+1$ is coprime to both $2$ and $3$.  Moreover, $s^2 + 1$ is not squarefree by Lemma \ref{lem:ghinelli} (iii).  Of the remaining possible values of $s$, there must be a prime $p$ dividing $s^2+1$ such that $(s^2 + 1)_p \ge 2s+3$ by Lemma \ref{lem:Opsize}, where $(s^2 + 1)_p$ denotes the highest power of $p$ dividing $s^2 + 1$.  Moreover, $(s^2 + 1)_p$ cannot be prime by Lemma \ref{lem:ghinelli} (ii), and $(s+1)(s^2 + 1)$ must divide $p^n \cdot \prod_{k=0}^{n-1}(p^{n-k} - 1)$ by Proposition \ref{prop:regregorder}.
\end{proof}

We next use Theorem \ref{thm:regregorder} to rule out many possible ``small'' groups.

\begin{prop}
 \label{prop:numerical}
 If $G$ is a group that acts regularly on both the point set and a line set of a thick generalized quadrangle of order $s$, then then $s > 10^8$ and $|G| > 10^{24} + 10^{16} + 10^8 + 1$.
\end{prop}

\begin{proof}
 The following computation was completed in $\GAP$ \cite{GAP}.  Such a generalized quadrangle of order $s$  must satisfy all the conditions imposed on $s$ listed in Theorem \ref{thm:regregorder}: namely, $s+1$ is coprime to both $2$ and $3$; $s^2 + 1$ is not squarefree; there is a prime $p$ dividing $s^2 + 1$ such that, if $p^n$ is the largest power of $p$ dividing $s^2+1$, then $n > 1$ and $p^n \ge 2s + 3$; and $(s+1)(s^2+1)$ divides $p^n \cdot \prod_{k=0}^{n-1}(p^{n-k} - 1)$.  Since no value of $s$ less than or equal to $10^8$ satisfies all of these conditions, the result follows.
\end{proof}

It would be interesting to know whether there exist any even values of $s$ that satisfy all of the restrictions mentioned in the statement of Theorem \ref{thm:regregorder}.

\section{Generalized quadrangles of order $s$, $s$ odd and $s+1$ coprime to $3$, with a group of automorphisms acting regularly on points}
\label{secss}

Let $\mathcal{Q}$ be a finite generalized quadrangle of order $(s,s)$, where $s>1$ is odd and $s+1$ is coprime to $3$, with point set $\mathcal{P}$ and line set $\mathcal{L}.$  Let $G$ be a group of automorphisms of $\mathcal{Q}$ that acts regularly on $\mathcal{P}$, and suppose that $N$ is a nonabelian minimal normal subgroup of $G$.  For a distinguished point of $P$ of $\mathcal{Q}$, we define $\Delta:= \{g \in G: P^g \sim P \} \cup \{1\}.$

First, by Lemma \ref{lem:yoshiara} (iv), $N \cong \Sz(q)^m$, where $m \in \mathbb{N}$ and $q = 2^{2e+1}$ for some $e \in \mathbb{N}.$  Let $N = S_1S_2...S_m$, where each $S_i \cong \Sz(q)$, each $S_i$ is normal in $N$, all $S_i$ and $S_j$ are conjugate in $G$, and $S_i \cap S_j = \{1\}$ for all $i\neq j$.

Since $|G| = |\mathcal{P}| = (s+1)(s^2+1)$ and $s^2 + 1 \equiv 2 \pmod 4,$ we have that the highest power of $2$ dividing $|G|$, which we will denote by $|G|_2$, is $2(s+1)_2$, i.e., two times the highest power of $2$ dividing $s+1$.  Since $|\Sz(q)|_2 = q^2$ by Lemma \ref{lem:suz} (i), $|N|_2 = q^{2m}$, and we immediately obtain the following inequality:
\begin{equation}
\label{eq:q2m}
 q^{2m} \le 2(s+1).
\end{equation}
Furthermore, by Lemma \ref{lem:yoshiara} (ii), since $\gcd(s,s) = s > 1,$ for all $a \in G$ we have $a^G \cap \Delta \neq \varnothing$.

\begin{lem}
\label{lem:directproduct}
Let $G$ be a group acting regularly on the set of points of a generalized quadrangle of order $(s,t)$, where $\gcd(s,t) > 1.$  Let $N = T_1T_2...T_k$ be a nonabelian minimal normal subgroup of $G$, where each $T_i$ is normal in $N$, has trivial intersection with the other $T_j$, and is isomorphic to the same nonabelian simple group $T$.  If each $T_i \subseteq \Delta$, then there exists a line $\ell$ such that $N \le G_{\ell}.$
\end{lem}

\begin{proof}
Assume that $T_i \subseteq \Delta$ for each $1 \le i \le k.$  By Lemma \ref{lem:yoshiara} (i) for each $i$ there is a (necessarily unique) line $\ell_i$ through the distinguished point $P$ such that $T_i \le G_{\ell_i}$.  Let $x_1 \in T_1$ and $x_2 \in T_2$.  Since $x_1$ and $x_2$ commute, both $P^{x_1}$ and $P^{x_2}$ are collinear with $P^{x_1x_2}$.  This implies that either $\ell_1 = \ell_2$ and $T_1T_2 \le G_{\ell_1}$ or that no element $x_1x_2,$ where neither $x_1$ nor $x_2$ is trivial, is in $\Delta.$  Let $x = x_1x_2 \in T_1T_2$ for some $x_1,x_2$ such that neither $x_1$ nor $x_2$ is trivial.  This means that $x^G$ consists of elements of the form $x_ix_j \in T_iT_j$, $1 \le i,j \le k.$  However, by Lemma \ref{lem:yoshiara} (ii), $x^G \cap \Delta \neq \varnothing$, so there exist $1 \le i_i, i_2 \le k$ such that $T_{i_1}T_{i_2} \le G_{\ell_{i_1}}$ for some line $\ell_{i_1}$ through $P.$  We now proceed by induction, and assume that $T_{j_1}T_{j_2}...T_{j_n} \le G_{\ell_{j_1}}$ for distinct indices $j_1,..., j_n.$ Arguing as above, since $\Delta$ meets every conjugacy class, there must be a set of indices $l_1,...,l_{n+1}$ such that $T_{l_1}T_{l_2}...T_{l_{n+1}} \le G_{l_1}$.  Hence $N = T_1...T_k \le G_{\ell}$ for some line $\ell$ through the point $P$, as desired. 
\end{proof}

\begin{lem}
\label{lem:m=1}
If $G$ and $N = S_1...S_m$ are as above, then $m=1$ and $N \cong \Sz(q)$ for some $q.$ 
\end{lem}

\begin{proof}
Applying Lemma \ref{lem:directproduct} to our situation, if each $S_i \subseteq \Delta$, then $N \le G_{\ell}$ for some line $\ell.$  This implies that $N \subseteq \Delta$, and by Lemma \ref{lem:yoshiara} (v) implies that $|N|$ divides $\gcd(s+1, s^2 - s) = 2$, a contradiction.  Hence we may assume without a loss of generality that $S_1 \cap \Delta^c \neq \varnothing.$  Let $a \in S_1 \cap \Delta^c.$  Clearly $a^G \cap \Delta^c \neq \varnothing,$ so \begin{equation}
\label{eq:s+1}
s+1 \le |a^G| 
\end{equation} by Lemma \ref{lem:xGc}. On the other hand, by Lemma \ref{lem:suz} (v) we have:
\begin{equation}
\label{eq:aG}
 |a^G| \le mq^2(q-1)(q + \sqrt{2q} + 1).
\end{equation}
Combining \eqref{eq:s+1} with \eqref{eq:q2m} and \eqref{eq:aG}, we find that $q^{2m} \le 2mq^2(q-1)(q + \sqrt{2q} + 1) < 2mq^5,$ and so $q^{2m-5} < 2m.$  Hence,
$$ 2m > q^{2m-5} \ge (1+7)^{2m-5} \ge 1 + 7(2m-5),$$
and, simplifying, we see that $m < \frac{34}{12} < 3,$ and so $m = 1$ or $2$.

Suppose that $m=2$, so $N \cong \Sz(q) \times \Sz(q)$.  The group $G$ must act transitively on the simple direct factors of $N$, so $2q^4$ divides $|G|_2$ and hence $q^4$ divides $(s+1)_2$.  This improves \eqref{eq:q2m} to $q^4 \le s+1.$  Going back to \eqref{eq:aG} and \eqref{eq:s+1}, we find now that: 
\begin{equation}
\label{eq:betters+1}
q^4 \le s+1 \le 2q^2(q-1)(q + \sqrt{2q} + 1) < 3q^4.
\end{equation} 
Let $p$ be any odd prime dividing $s+1$.  Then $pq^4$ divides $s+1$, and, combined with \eqref{eq:betters+1}, this yields $5q^4 \le pq^4 \le s+1 < 3q^4$, a contradiction.  

Hence $s+1$ must be a power of $2$, and $s+1 = |G|_2/2$.  Moreover, since $q^4 \le s+1 < 3q^4$ by \eqref{eq:q2m}, $s+1$ must be either $q^4$ or $2q^4$.  Since $|\Sz(q)|$ divides $|G|$, $q-1$ divides $|G|$ by Lemma \ref{lem:suz} (i).  Since $s+1$ is a power of $2$, this means that $q-1$ must divide $s^2+1$.  However, if $s+1 = q^4$, then \[s^2 + 1 = (q^4 - 1)^2 + 1 = (q-1)\frac{(q^4-1)^2}{q-1} + 1,\] and, if $s+1 = 2q^4$, then \[s^2 + 1 = (2q^4 - 1)^2 + 1 = (q-1)\frac{4q^4(q^4-1)}{q-1} + 2,\] and so $q-1$ does not divide $s^2+1$ in either case, a contradiction.  Therefore, $m=1$, and $N \cong \Sz(q)$ for some $q$.

\end{proof}

\begin{lem}
\label{lem:unique}
If $G$ is as above, then it has a unique nonabelian minimal normal subgroup $N$. 
\end{lem}

\begin{proof}
Suppose that $G$ has two minimal normal subgroups, $M$ and $N$.  By Lemma \ref{lem:m=1}, $N \cong \Sz(q)$ and $M \cong \Sz(q')$, where without a loss of generality $q \le q'.$  Distinct minimal normal subgroups must commute, and so we may proceed as in the proof of Lemma \ref{lem:m=1}: $q^2(q')^2/2$ divides $s+1$, and so $q^4/2 \le q^2(q')^2/2 \le s+1.$  Also, $N \cap \Delta^c \neq \varnothing$ by Lemma \ref{lem:yoshiara} (v), so for some $x \in N$, $x^G \cap \Delta^c \neq\varnothing$.  Thus, by Lemma \ref{lem:xGc} and Lemma \ref{lem:suz} (v),
\[s+1 \le |x^G| \le q^2(q-1)(q + \sqrt{2q} + 1) < 2q^4.\]   
If an odd prime $p$ divides $s+1$, then $p \ge 5$, and
\[ s+1 \ge p(s+1)_2 \ge \frac{5q^4}{2} > 2q^4,\] a contradiction, and so no odd prime can divide $s+1$.  However, if $s+1$ is a power of $2$, then $s+1 = |G|_2/2$, and, since $q^4$ divides $|G|$, $q^4/2$ divides $s+1$.  We know from above that $s+1 < 2q^4$,  so either $s+1 = q^4/2$ or $s+1 = q^4$, since $s+1$ is a power of $2$.  Moreover, if $s+1$ is a power of $2$, then, since $q-1$ divides $|N|$, $q-1$ must divide $|G|$ and hence $s^2 + 1$.  If $s+1 = q^4$, then
\[s^2 + 1 = (q^4 - 1)^2 + 1 = (q-1)\frac{(q^4-1)^2}{q-1} + 1.\] If $s+1 = q^4/2$, then $q-1$ divides $4(s^2 + 1)$, and so
\[4(s^2 + 1) = q^8 -4q^4 + 8 = (q-1)(q^7 + q^6 +q^5 +q^4 - 3q^3 - 3q^2 - 3q - 3) + 5,\]
and hence $q-1$ cannot divide $s^2 + 1$ in either case, a contradiction.  Therefore, if $G$ has a nonabelian minimal normal subgroup, it must be unique.
\end{proof}

\begin{lem}
\label{lem:2}
If $G$ is as above, then $s+1$ cannot be a power of $2$. 
\end{lem}

\begin{proof}
We know that $N \cong \Sz(q)$ is the unique minimal normal subgroup of $G$, and so $q^2$ divides $|G|_2 = 2(s+1)_2$, i.e., $\frac{q^2}{2}$ divides $s+1$.  Let $s+1 = 2^{n-1}q^2$ for some $n \in \mathbb{N}.$  Hence $s^2 + 1 = 2^{2n-2}q^4 - 2^nq^2+2.$

Now, since $(q-1)(q^2+1)$ divides $|\Sz(q)|$, $(q-1)(q^2+1)$ must divide $s^2+1 = 2^{2n-2}q^4 - 2^nq^2+2.$  Now, \[2^{2n-2}q^4 - 2^nq^2+2 = (q^2+1)(2^{2n-2}q^2 - (2^{2n-2}+2^n)) + (2^{n-1} + 1)^2 + 1,\] which implies that $q^2 + 1$ must divide $(2^{n-1} + 1)^2 + 1$, which in turn implies that $q \le 2^{n-1} + 1$.  Since $q$ is a power of $2$, we have $q \le 2^{n-1}$.

Note also that $\gcd(q^2+1,s+1) = 1$. Let $p$ be an odd prime dividing $q + \sqrt{2q} + 1,$ and let $a \in N$ be an element of order $p$ (Sylow subgroups of odd order are cyclic in $\Sz(q)$ by Lemma \ref{lem:suz} (iii)).  Since $p$ does not divide $s+1$, by Lemma \ref{lem:yoshiara} (vi) we have $a^G \cap \Delta^c \neq \varnothing$, and so, by Lemma \ref{lem:xGc},
\begin{equation}
\label{eq:news+1}
s+1 \le |a^G| = q^2(q-1)(q-\sqrt{2q} + 1) < q^4.
\end{equation}
Furthermore, since $(q-1)(q^2+1) = \frac{q^4 - 1}{q+1}$ and \[s^2 + 1 = 2^{2n-2}q^4 - 2^nq^2+2 = (q^4-1)2^{2n-2} + (2^{2n-2} - 2^nq^2 + 2),\] $(q-1)(q^2 +1)$ must divide $2^{2n-2} - 2^nq^2 + 2$.  We divide into two cases, depending on whether $2^{2n-2} - 2^nq^2 + 2$ is nonnegative or negative.

Suppose first that $2^{2n-2} - 2^nq^2 + 2 \ge 0$.  Thus $2^nq^2 \le 2^{2n-2}+2,$ and so $q^2 \le 2^{n-2} + \frac{1}{2^{n-1}}$, and, since $q$ is a power of $2$, $q^2 \le 2^{n-2}.$  But then $s+1 = 2^{n-1}q^2 \ge 2q^4,$ contradicting \eqref{eq:news+1}.

Now suppose that $2^{2n-2} -2^nq^2 + 2 < 0$.  Thus $(q-1)(q^2+1)$ divides $2^nq^2 - 2^{2n-2} - 2$, which is a positive integer.  We already know that $q \le 2^{n-1}$, so let $2^{n-1} = 2^mq$ for some $m \ge 0.$  Substituting, we have $(q-1)(q^2+1)$ divides $2^{m+1}q^3 - 2^{2m}q^2 - 2$.  Since \[2^{m+1}q^3 - 2^{2m}q^2 - 2 = (q-1)(q^2+1)2^{m+1} + (2^{m+1} - 2^{2m})q^2 - 2^{m+1}q + (2^{m+1}-2),\] $(q-1)(q^2 + 1)$ must divide $(2^{m+1} - 2^{2m})q^2 - 2^{m+1}q + (2^{m+1}-2)$, and since \[(2^{m+1} - 2^{2m})q^2 - 2^{m+1}q + (2^{m+1}-2) = (2^{m+1} - 2^{2m})(q^2+1) + (2^{2m} -2^{m+1}q -2),\] $q^2+1$ must divide $2^{2m} -2^{m+1}q -2$.  We once more split into two cases, depending on the sign of $2^{2m} -2^{m+1}q -2$.

If $2^{2m} -2^{m+1}q -2 \ge 0$, then $2^{m+1}q \le 2^{2m} - 2$, and so $q < 2^{m-1}$.  On the other hand, if $2^{2m} -2^{m+1}q -2 < 0$, then $q^2+1 \le 2^{m+1}q - 2^{2m} + 2$, which implies that $q^2 \le 2^{m+1}q$ and that $q \le 2^{m+1}$.  Moreover, if $q = 2^{m+1}$ in this last case, then $m = 0$, which is impossible, since $q \ge 8$.  Thus, in either case, $q \le 2^m$.  Recalling that $s+1 = 2^{n-1} q^2$ and $2^{n-1} = 2^m q$, we have 
$$(s+1) = 2^{n-1}q^2 = 2^mq^3 \ge q^4,$$
which contradicts \eqref{eq:news+1}.  Therefore, $s+1$ cannot be a power of $2$.
\end{proof}

In order to further determine the structure of $G$, we will now examine the generalized Fitting subgroup of the centralizer in $G$ of $N$, $F^*(C_G(N)) = E(C_G(N))F(C_G(N))$.  Since $N$ is the unique nonabelian minimal normal subgroup and the Schur multiplier of $\Sz(q)$ is either trivial or, if $q=8$, an elementary abelian group of order $4$ by Lemma \ref{lem:suz} (vii), $E(C_G(N)) \le E(G) = N$ or $E(C_G(N)) \le E(G) \le 2^2 \cdot N$, and so $E(C_G(N)) = 1$ (as in \cite[Step 5]{yoshiara}).  This means that $F^*(C_G(N)) = F(C_G(N))$.  Note further that since $F(C_G(N))$ is characteristic in $C_G(N)$, which is itself normal in $G$, $F(C_G(N)) \norml G$, which implies that $F(C_G(N)) \le F(G).$

\begin{lem}
 \label{lem:OpG}
 There is at most one odd prime $p$ such that $O_p(G) \neq 1$.  Moreover, if such an odd prime $p$ exists, then the following hold:
 \begin{itemize}
  \item[(i)] $p$ divides $s^2 + 1$;
  \item[(ii)] for any nonidentity element $x \in G$, $|x^G| \ge s+1$;
  \item[(iii)] $|O_p(G)| \ge s+2$.
 \end{itemize}

\end{lem}

\begin{proof}
The proof of this lemma is analagous to the proof of \cite[Theorem 4.2]{ghinelli}.  By Lemma \ref{lem:ghinelli} (i), $O_p(G) = 1$ for all odd primes dividing $s+1$.  Let $p_1,p_2$ be two odd primes dividing $s^2+1$, and let $N_1 = O_{p_1}(G)$ and $N_2 = O_{p_2}(G)$.  Without a loss of generality, $|N_1| = p_1^{\alpha_1} \ge p_2^{\alpha_2} = |N_2|.$  Note that $p_2^{\alpha_2} < s+1,$ since otherwise $(s+1)^2 \le |N_1||N_2| \le s^2+1$, a contradiction.  On the other hand, let $1 \neq x \in N_2.$  Since $|x^G| \le |N_2| -1 <s$, $|x^G \cap \Delta^c| < s.$  Since $s$ divides $|x^G \cap \Delta^c|$ by Lemma \ref{lem:yoshiara} (ii), this implies that $x^G \cap \Delta^c = \varnothing$ and $x^G \subseteq \Delta.$  But then by Lemma \ref{lem:yoshiara} (vi) $p_2$ divides $s+1$, a contradiction.  Hence there is at most one odd prime $p$ such that $p$ divides $|F(G)|$.  

Let $|O_p(G)| = p^\alpha$ and $1 \neq x \in O_p(G)$.  Since $p$ does not divide $s+1$, $x^G \cap \Delta^c \neq \varnothing$ by Lemma \ref{lem:yoshiara} (vi), and so $|x^G| \ge s+1$ by Lemma \ref{lem:xGc}.  This proves (ii).  Finally, $1 \in O_p(G)$ by definition and so
\[ |O_p(G)| \ge |1^G| + |x^G| \ge s+2.\]
\end{proof}

\begin{lem}
 \label{lem:O2G}
  If $|O_2(G)| \le s+1$, then $O_2(G) \liso C_2$, the cyclic group of order $2$.
\end{lem}

\begin{proof}
By assumption, $|O_2(G)| \le s+1$.  In particular, this implies that for each $x \in O_2(G)$, $|x^G| < s$.  This means that $x^G \subseteq \Delta$ by Lemma \ref{lem:yoshiara} (ii), which in turn implies that $O_2(G) \subseteq \Delta.$  However, by Lemma \ref{lem:yoshiara} (v), this means that $|O_2(G)|$ divides $\gcd(s+1,s^2-s) =2.$  Therefore, $O_2(G) \liso C_2$, as desired. 
\end{proof}

\begin{lem}
\label{lem:fitting}
The Fitting subgroup $F(C_G(N)) \liso C_2$. 
\end{lem}

\begin{proof}
Let $p$ be an odd prime that divides $|F(C_G(N))|,$ and let $O_p(C_G(N)) = Q$. Since $Q$ is characteristic in $F(C_G(N))$, $Q \norml G.$  Let $x \in Q$, $x \neq 1$.  By Lemma \ref{lem:OpG} (i), we have $p \mid (s^2 + 1)$, and so $x^G \cap \Delta^c \neq \varnothing$ by Lemma \ref{lem:yoshiara} (vi).  Thus, by Lemma \ref{lem:xGc},
\[ |Q| \ge |1^G| + |x^G| \ge s+2.\]
On the other hand, since $\gcd(s+1, s^2 - s) = 2$, by Lemma \ref{lem:yoshiara} (v) we know that $N$ is not contained entirely in $\Delta$, and so for some $a \in N$, $a^G \cap \Delta^c \neq \varnothing$, and hence, by Lemma \ref{lem:xGc},
\[ |N| \ge |1^G| + |a^G| \ge s + 2.\]  Let $1 \neq x \in Z(Q).$  Then $|C_G(x)| \ge |Q||N| \ge (s+2)(s+2).$  Hence, we have:
$$|x^G| = \frac{|G|}{|C_G(x)|} \le \frac{(s+1)(s^2+1)}{(s+2)(s+2)} < s.$$
However, this implies that $x^G \subseteq \Delta$ by Lemma \ref{lem:yoshiara} (ii), and by Lemma \ref{lem:yoshiara} (vi) we have that $p$ divides $s+1$, a contradiction.  Therefore, $F(C_G(N)) = O_2(C_G(N)) \le O_2(G)$.

Finally, we know that $|O_2(G)| \le 2(s+1)_2$.  By Lemma \ref{lem:2}, $s+1$ cannot be a power of $2$, so there is an odd prime $r$ that divides $s+1$.  This implies that
\[ |O_2(G)| \le 2(s+1)_2 < r(s+1)_2 \le s+1.\]  By Lemma \ref{lem:O2G}, we have $O_2(G) \liso C_2$, and so $F(C_G(N)) \liso C_2$, as desired. 
\end{proof}

\begin{thm}
\label{thm:s}
Let $s >1$ be an odd integer such that $s+1$ is coprime to $3$.  If $G$ acts regularly on the point set of a generalized quadrangle of order $s$, then $G$ has no nonabelian minimal normal subgroups. 
\end{thm}

\begin{proof}
By Lemma \ref{lem:fitting}, $F^*(C_G(N)) = F(C_G(N)) \liso C_2.$  The generalized Fitting subgroup is normal in $C_G(N)$ and self-centralizing; however, this is only possible if $C_G(N) \liso C_2$ itself.

First, suppose that $C_G(N) \cong C_2.$  Note that $G/(N C_G(N)) \liso \Out(\Sz(q)) \cong C_{2e+1}$, where $q = 2^{2e+1}$ (see Lemma \ref{lem:suz} (viii)).  As in the proof of \cite[Step 5]{yoshiara}, this means that: 
$$|G| \le 2(2e+1)|\Sz(q)| < 2q^6.$$
On the other hand, $2q^2$ divides $|G|_2 = 2(s+1)_2$, so $q^2$ divides $(s+1)_2.$  By Lemma \ref{lem:2}, $s+1$ is not a power of $2$, so $5q^2 \le s+1.$  But then: 
$$|G| = (s+1)(s^2 + 1) \ge 5q^2 ((5q^2 - 1)^2 + 1) > 124q^6,$$
a contradiction.

Now suppose that $C_G(N) = 1.$  Then $G/N \liso \Out(\Sz(q)) \cong C_{2e+1}$.  Arguing as above, this implies that 
\[|G| \le (2e+1)|\Sz(q)| < q^6.\]  
On the other hand, $q^2$ divides $|G|_2 = 2(s+1)_2$, and so, again proceeding as above, since $s+1$ is not a power of $2$ and $3 \nmid (s+1)$, this implies that $\frac{5q^2}{2} \le s+1$, which would in turn imply that 
$|G| > 6q^6$, a contradiction.  Therefore, if $G$ acts regularly on the points of such a generalized quadrangle of order $(s,s)$, then $G$ does not have a nonabelian minimal normal subgroup.
\end{proof}

We will now assume, like above, that $s$ is odd, $3$ does not divide $s+1$, and $G$ acts regularly on the set of points of a generalized quadrangle of order $(s,s).$  We proved above that such a $G$ cannot have a nonabelian minimal normal subgroup, and the following lemma summarizes what can be said in the case that $G$ is not solvable.

\begin{lem}
\label{lem:notsolv}
If $G$ is not solvable, then $F^*(G) \liso O_p(G) \times C_2$, where $p$ is an odd prime dividing $s^2+1$ and for any $1 \neq x \in O_p(G)$, $|x^G| \ge s+1$.  
\end{lem}

\begin{proof}
If such a group $G$ exists, then every minimal normal subgroup of $G$ is abelian by Theorem \ref{thm:s}, and so $F^*(G) = F(G)$.  By Lemma \ref{lem:OpG}, there is at most one odd prime $p$ for which $O_p(G) \neq 1$, and, if such a $p$ exists, then $p$ divides $s^2 + 1$ and $|x^G| \ge s+1$ for all nonidentity elements $x \in O_p(G)$.  Thus $F^*(G) \liso O_p(G) \times O_2(G)$, and it only remains to show that $O_2(G) \cong C_2$.

Suppose that $|O_2(G)| > s+1$.  Since $|O_2(G)| \le |G|_2 = 2(s+1)_2$, this would require that $s+1$ is a power of $2$, in which case $|O_2(G)| = 2(s+1)$, and then $O_2(G)$ is a Sylow 2-subgroup of $G$.  Thus $G/O_2(G)$ has odd order and is solvable. Since both $O_2(G)$ and $G/O_2(G)$ are solvable, so is $G$, contrary to our hypotheses.  Therefore, if $G$ is not solvable, $|O_2(G)| \le s+1$, which means $O_2(G) \liso C_2$ by Lemma \ref{lem:O2G}, completing the proof.
\end{proof}

We remark that, since $\gcd(s+1, s^2 - s) = 2$, by Lemma \ref{lem:yoshiara} (v), if $F^*(G)$ has even order, then $Z(G) \cong C_2.$  Otherwise, $F^*(G) \cong O_p(G)$ and $Z(G) = 1.$

It would be interesting to rule out entirely groups acting regularly on the point set of a generalized quadrangle of order $s$, where $s$ is odd and $s+1$ is coprime to $3$, in the same way that Yoshiara ruled out such groups when $s = t^2$.  A somewhat minor difference is that we cannot rule out the possibility that $O_2(G) \neq 1$.  A more substantial difference here is that, since $|\mathcal{P}| = |\mathcal{L}|$, we cannot guarantee that any lines are fixed by the elements of odd order of $G$, which is crucial step in Yoshiara's proof (see \cite[Step 7, Final step]{yoshiara}).

\section{Generalized quadrangles of order $(u^2,u^3)$ with a group of automorphisms acting regularly on points}
\label{sect:u^2}

Throughout this section, let $\mathcal{Q}$ be a generalized quadrangle of order $(u^2,u^3)$, $u>1$, with point set $\mathcal{P}$ and line set $\mathcal{L}$.  Let $G$ be a group of automorphisms of $\mathcal{Q}$ that acts regularly on $\mathcal{P}.$  Note that in this case $|G| = |\mathcal{P}| = (u^2 + 1)(u^5+1)$.

\begin{lem}
\label{lem:u3}
If $G$ is as above, then $3$ does not divide $u+1$. 
\end{lem}

\begin{proof}
Suppose that $u \equiv -1 \pmod 3$.  This means that $u^5 + 1 \equiv 0 \pmod 3$, and so $3$ divides $|G|.$  However, by Lemma \ref{lem:yoshiara} (iii), this implies that $3$ divides $u^2+1$.  However, $3$ never divides $u^2+1$ for any integer $u$, a contradiction.  Hence no group of automorphisms can act regularly on $\mathcal{P}$ in this instance.
\end{proof}

Assume henceforth that $3$ does not divide $u+1$.  Note that if $u$ is even, then $|G|$ is odd, and so $G$ would be solvable.  We thus assume that $u$ is odd.  Suppose that $G$ has a nonabelian minimal normal subgroup $N$.  Since $3$ does not divide $|G|$, by Lemma \ref{lem:yoshiara} (iv), we have $N \cong \Sz(q)^m$ for some $m \ge 1$ and $q = 2^{2e+1} \ge 8.$  Let $N = S_1S_2...S_m$, where each $S_i$ (i) is normal in $N$, (ii) is conjugate in $G$ to and has trivial intersection with all other $S_j$, and (iii) is isomorphic to $\Sz(q)$. 

First, $u^2 + 1 \equiv 2 \pmod 4,$ so $|G|_2 = 2(u^5 + 1)_2.$  Moreover, $u^5 + 1 = (u+1)(u^4 - u^3 + u^2 - u +1)$, and $u^4 - u^3 + u^2 - u +1$ is odd, so $|G|_2 = 2(u+1)_2.$  Hence, by Lemma \ref{lem:suz} (i),
\begin{equation}
 \label{eq:uq2m}
|N|_2 = q^{2m} \le 2(u+1)_2 \le 2(u+1).
\end{equation}
Let $x$ be an element of $S_1$ of order $4$.  If $x^G \subseteq \Delta$, then by Lemma \ref{lem:yoshiara} (vi), $4$ would divide $u^2+1$, a contradiction.  Hence $x^G \cap \Delta^c \neq \varnothing,$ and by Lemma \ref{lem:xGc}, this means $|x^G| \ge u^2 + 1$.  Since the order of the centralizer of an element of order $4$ in $\Sz(q)$ is $2q$ by Lemma \ref{lem:suz} (vi), we have: 
\begin{equation}
\label{eq:uxG}
u^2 + 1 \le |x^G| = \frac{1}{2}mq(q-1)(q^2+1) < \frac{1}{2}mq^4.
\end{equation}

\begin{lem}
\label{lem:uSz}
If such a $G$ and $N$ exist, then $m=1$ and $N \cong \Sz(q)$. 
\end{lem}

\begin{proof}
From \eqref{eq:uq2m} and \eqref{eq:uxG}, we have
$$\left(\frac{q^{2m}}{2} - 1\right)^2 + 1 \le u^2 + 1 < \frac{1}{2}mq^4.$$
Simplifying, this yields $q^{4m} < 2mq^4 + 4q^{2m}$.  If $m \ge 2$, this means that 
$$q^{4m} <  2mq^4 + 4q^{2m} \le 2mq^{2m} + 4q^{2m} \le (2m + 4)q^{2m},$$ and so:
$$2m+4 > q^{2m} \ge (1+7)^{2m} > 1 + 7(2m),$$
which simplifies to $m < \frac{1}{4},$ a contradiction.  Therefore, $m=1$, and $N \cong \Sz(q)$ for some $q$.

\end{proof}

\begin{lem}
\label{lem:u2}
If such a $G$ and such an $N$ as above exist, then $u+1$ cannot be a power of $2$. 
\end{lem}

\begin{proof}
Suppose that $u+1$ is a power of $2$.  First, note that \eqref{eq:uq2m} becomes $q^2 \le 2(u+1)$ when $m=1$, which implies that $q^4 \le 4(u+1)^2.$  Note also that \eqref{eq:uxG} becomes $2(u^2 + 1) < q^4$ when $m=1$.  Putting these together, we see that:
$$(u+1)^2 < 2(u^2+1) < q^4 \le 4(u+1)^2,$$
which implies that $(u+1) < q^2 \le 2(u+1).$  Since $u+1$ is a power of $2$ by assumption, this means that $q^2 = 2(u+1).$  Hence $q^2 + 1 = 2u+3$.  Since $q^2+1$ divides $|N|$, $q^2 + 1$ divides $|G|$.  Since $q^2+1$ divides $|G|$ and $q^2+1 = 2u+3$ and is odd, $2u+3$ must divide 
\[ \frac{|G|}{u+1} = \frac{(u^2 + 1)(u^5 + 1)}{u+1} = (u^2+1)(u^4 - u^3 + u^2 - u + 1) = u^6 - u^5 + 2u^4 - 2u^3 + 2u^2 - u + 1.\]  On the other hand,
$$ 64(u^6 - u^5 + 2u^4 - 2u^3 + 2u^2 - u + 1) = (2u+3)(32u^5 - 80u^4 + 184u^3 - 340u^2 + 574u - 893) + 2743,$$
so $2u+3$ must divide $2743 = 13 \cdot 211$.  This means that $2u+3$ is one of $13, 211, 2743$, which means that $u+1$ is one of $6, 105, 1371,$ none of which is a power of $2$, a contradiction.  Therefore, $u+1$ cannot be a power of $2$.
 
\end{proof}

\begin{thm}
\label{thm:u^2}
If $G$ acts regularly on the set of points of a generalized quadrangle of order $(u^2, u^3)$, then $G$ cannot have a nonabelian minimal normal subgroup. 
\end{thm}

\begin{proof}
By Lemmas \ref{lem:u3}, \ref{lem:uSz}, and \ref{lem:u2}, $3$ does not divide $u+1$, $u+1$ is not a power of $2$, and $N \cong \Sz(q)$ for some $q = 2^{2e+1}$.  Let $r$ be an odd prime dividing $u+1$.  Then, noting that $m = 1$, from \eqref{eq:uq2m} we have
$$q^2 = |N|_2 \le |G|_2 = 2(u+1)_2 < r(u+1)_2 \le u+1,$$
which implies that $q^4 \le u^2 + 2u + 1$.  On the other hand, again noting that $m = 1$, from \eqref{eq:uxG} we have:
$$u^2 + 1 < \frac{1}{2}q^4.$$
Putting these together, we see that $2u^2 + 2 < u^2 + 2u + 1$, which simplifies to $(u-1)^2 < 0$, a contradiction.  Therefore, no such $N$ can exist, and such a $G$ cannot have a nonabelian minimal normal subgroup.  
\end{proof}

Note that Theorem \ref{thm:minnormalsubs} follows immediately from Theorems \ref{thm:s} and \ref{thm:u^2}.

We will now assume, like above, that $u$ is odd, $3$ does not divide $u+1$, and $G$ acts regularly on the set of points of a generalized quadrangle of order $(u^2,u^3).$  We proved above that such a $G$ cannot have a nonabelian minimal normal subgroup, and the following lemma summarizes what can be said in the case that $G$ is not solvable.

\begin{lem}
If $G$ is not solvable, then $F^*(G) \liso O_p(G) \times C_2$, where $p$ is an odd prime dividing $u^4 - u^3 + u^2 - u + 1$ and for any $1 \neq x \in O_p(G)$, $|x^G| \ge u^2+1$.  
\end{lem}

\begin{proof}
The proof proceeds similarly to those in Section \ref{secss}.  Assume that $G$ is not solvable.

\begin{itemize}
 \item[(Step 1)] \textit{If there is a normal subgroup $N$ of $G$ such that $N \subseteq \Delta$, then $|N|$ divides $\gcd(u^2 +1, (u^3)^2 - u^3) = 2$.}\\  This follows from Lemma \ref{lem:yoshiara} (v).
 \item[(Step 2)] \textit{If $p$ is an odd prime and $O_p(G)$ is nontrivial, then $p$ divides $u^4 - u^3 + u^2 - u + 1$ and $|O_p(G)|\ge u^2 + 2$.}\\
 Let $p$ be an odd prime such that $O_p(G)$ is nontrivial.  By (Step 1), $O_p(G)$ cannot be entirely contained in $\Delta$.  This means that there is $x \in O_p(G)$ such that $x \in \Delta^c$, and so by Lemma \ref{lem:xGc}, $|x^G| \ge u^2 + 1$.  Thus 
 \[ |O_p(G)| \ge |1^G| + |x^G| \ge u^2 + 2.\]
 Since $\gcd(u^2 + 1, u^5 + 1) = 2$ and $|O_p(G)| > u^2 + 1$, we have that $p$ divides $u^5 + 1$.  Moreover, since $\gcd(u+1, u^4 - u^3 + u^2 - u + 1) = 1$, we have that $p$ divides $u^4 - u^3 + u^2 - u + 1$.
 \item[(Step 3)] \textit{There is at most one odd prime $p$ such that $O_p(G)$ is nontrivial.}\\
 Assume that $O_{p_1}(G)$ and $O_{p_2}(G)$ are nontrivial for distinct odd primes $p_1$ and $p_2$.  By (Step 2), both $p_1$ and $p_2$ divide $u^4 - u^3 + u^2 - u + 1$ and both $|O_{p_1}(G)|,$ $|O_{p_2}(G)|$ are at least $u^2 + 2$.  However, since $p_1$ and $p_2$ are distinct,
 \[u^4 - u^3 + u^2 - u + 1 \ge |O_{p_1}(G)||O_{p_2}(G)| \ge (u^2 + 2)^2,\]
 which simplifies to $(u+3)(u^2 + 1) \le 0$, a contradiction.
 \item[(Step 4)] \textit{$O_2(G) \liso C_2$}\\
 Since $|O_2(G)| \le |G|_2 = 2(u+1)_2$, $|O_2(G)| \le 2(u+1) < u^2 + 1$, and so $|x^G| < u^2$ for all $x \in O_2(G)$.  By Lemma \ref{lem:yoshiara} (ii), this means $x^G \subseteq \Delta$, and so $O_2(G) \subseteq \Delta$.  The result follows from (Step 1).
 \item[(Step 5)] \textit{$F^*(G) \liso O_p(G) \times C_2$, where $p$ is an odd prime dividing $u^4 - u^3 + u^2 - u + 1$}\\
 By Theorem \ref{thm:u^2}, $G$ has no nonabelian minimal normal subgroups, and hence $F^*(G) = F(G)$.  This now follows from (Step 3) and (Step 4).
 \item[(Step 6)]  \textit{For all nonidentity elements of $O_p(G)$, $|x^G| \ge u^2 + 1$.}\\
 Let $x \in O_p(G)$, $x \neq 1$.  Since $p \nmid u^2 + 1$, by Lemma \ref{lem:yoshiara} (vi) we know that $x^G$ is not contained in $\Delta$.  The result follows from Lemma \ref{lem:xGc}.
 \end{itemize}
\end{proof}

We remark that, much as in the case when $s = t$ in Section \ref{secss}, it would be nice to rule out the case when $(s,t) = (u^2,u^3)$ entirely.  While we can guarantee here that for every odd prime $r$ dividing $u^2 + 1$, there are at least two lines stabilized by a nontrivial $r$-subgroup $R$, we cannot guarantee that \textit{exactly} two lines are stabilized by $R$, which was the conclusion of \cite[Step 7]{yoshiara}.  (Yoshiara used \cite[1.2.4]{fgq}, which is a result specific to generalized quadrangles with $s = t^2$.)  Moreover, even if we know that exactly two lines are stabilized by $R$ (which would force $R$ to act coprimely on $O_p(G)$ as in \cite[Final step]{yoshiara}), we have that 
\[ u^4 - u^3 + u^2 - u + 1 \equiv 1 \pmod {u^2 + 1},\] and so such an $r$-subgroup could in theory act coprimely on $O_p(G)$.

\vspace{.3cm}

\noindent\textsc{Acknowledgements.}  The author would like to thank John Bamberg and Cai-Heng Li for many interesting conversations on this topic as well as feedback on drafts of this manuscript, and the author would also like to thank the anonymous referees for many excellent suggestions that greatly improved the readability of this paper.  This work was initially started when the author was employed at the University of Western Australia, and the author acknowledges the support of the Australian Research Council Discovery
Grant DP120101336 during his time there.

\section{\refname}
\bibliographystyle{plain}
\bibliography{RegGQrev}

\end{document}